\documentclass[12pt]{amsart}
\usepackage{latexsym,fancyhdr,amssymb,color,amsmath,amsthm,graphicx,listings,comment}
\usepackage[section]{placeins}
\newtheorem{thm}{Theorem} 
 
\newtheorem{lemma}{Lemma} 
\newtheorem{coro}{Corollary}
\newtheorem{conj}{Conjecture}
\let\paragraph\subsection

\title{On Natural Groups}
\author{Oliver Knill}
\date{9/7, 2026}
\address{Department of Mathematics \\ Harvard University \\ Cambridge, MA, 02138 }


\begin{document}
\maketitle

\begin{abstract}
A group $(G,*)$ is natural \cite{GraphsGroupsGeometry} if its group operation is uniquely 
determined by a metric structure $(G,d)$ on $G$ in the following sense: every group
structure $(G,\circ)$ for which all right translations are isometries
of $(G,d)$ must be isomorphic to $(G,*)$. Examples included all connected Lie groups or
all groups generated by involutions \cite{Sutherland2026}. The orientation rigidity theorem of 
\cite{LeemannDeLaSalle2021} allows to upgrade the non-abelian structure theorem:
every non-abelian group of cardinality $|G| \leq c$ that is not generalized 
dicyclic is natural. A consequence is that all non-metabelian Lie groups, all simple
group of cardinality $|G| \leq c$, 
all homeomorphism-, diffeomorphism -or symplectomorphism groups of manifolds, or 
automorphism groups probability spaces or non-abelian crystallographic groups 
are natural. Also the abelian structure theorem is extended:
affine rigidity of Jarosz \cite{Jarosz1988}, together with Mazur-Ulam's theorem implies
that the additive group of every real or complex vector spaces is natural and that
all connected abelian Banach Lie groups are natural. A theorem of Babai \cite{Babai1978}
implies that every Boolean group $C_2^{(k)}$ is natural for every cardinal $k$.
Undecided is whether $C_p^{(k)}$ is natural for cardinals $k>c$ and odd prime $p$,
and whether the additive group of any field $F$ is natural, or whether the cardinality 
assumption in the non-abelian structure theorem is needed. A major question is whether
$G^2=G$ implies that $G$ is natural. This is already open for abelian groups:
does $2G=G$ imply that $G$ is natural?
\end{abstract}

\section{Introduction}

\paragraph{}
Within topological structures, {\bf metric spaces} are particularly important.
They balance generality, simplicity and intuition. Although topological spaces
are more general than metric spaces, many of the spaces that appear in geometry and 
analysis are metrizable. Moreover, isometries are considerably more rigid on 
metric spaces than homeomorphisms. The group of isometries of a group equipped 
with a metric has a chance to be of comparable size than $G$.

\paragraph{}
Among algebraic structures with a single operation, {\bf groups} play a
central role. While monoids, magmas, quasigroups, loops, groupoids or 
semigroups are more general, symmetries in applications are 
overwhelmingly modeled by groups. Klein's Erlangen program made
the relationship between groups and geometry more explicit.
For continuous symmetry groups, Noether's theorem links groups with 
conservation laws.

\paragraph{}
What happens if the same set $G$ carries both a group structure $(G,*)$ as
well as a metric structure $(G,d)$? 
We called a group {\bf natural} \cite{GraphsGroupsGeometry}, 
if there exists a metric $(G,d)$ such that all right 
translations are isometries and, up to group isomorphism, 
no other group structure on the same set has this property. 
For a natural group, its multiplication table is determined by geometry.

\paragraph{}
A group of finite prime order $p$ is natural because Lagrange's theorem implies
that there is only one group of order $p$.
The group $C_2^r=C_2 \times C_2 \times \cdots \times C_2$ is natural because it 
is a reflection group. Intuitively, it can be see to be natural because it is the
symmetry group of a discrete $r$-cube for which all sides have different lengths. 
Cyclic groups $C_{2k+1}$ are natural because we know the automorphism group 
of the cyclic graph. 

\paragraph{}
A few less obvious mechanisms allow to prove that group is natural. $(\mathbb{R}^n,+)$  for example is 
natural because it is a connected Lie group. It can also be seen to be natural because
it is the additive group of a vector space. The Euclidean group $E(n)$ in $n>0$-dimensions
is natural because it is generated by involutions. The diffeomorphism group ${\rm Diff}(\mathbb{T}^1)$ 
of the circle is natural because its commutator subgroup is simple. 
This prevents it to be generalized dicyclic and forces it to be natural by a structure theorem. 

\paragraph{}
Similarly, there are a few mechanisms known for a group to be non-natural.
The integer group $(\mathbb{Z},+)$ is not natural: let $d$ be any translation 
invariant metric. Then $T(x)=x+1$ is an isometry. It follows that $S(x)=-x$ is also
an isometry because $d(-x,-y)=d(0,x-y)=d(y,x)=d(x,y)$ by the 
symmetry axiom $d(x,y)=d(y,x)$ of $d$. The two maps $a(x)=x+2$ and $b(x)=1-x$ are
both isometries and $b^2=1, b a b=a^{-1}$, so that the action of 
$\langle a,b \rangle \cong D_{\infty}$ is {\bf regular} (transitive and free):
{\bf transitivity} holds as $a$ preserves parity classes while $b$ interchanges it.
It is {\bf free} because $a^k(x)=x$ implies $k=0$ while
$a^k b(x)=x$ would give $2x=2k+1$ which is impossible in $\mathbb{Z}$.
Every translation-invariant metric on $\mathbb{Z}$ therefore also admits the 
non-abelian, and so non-isomorphic, group $D_{\infty}$ as a regular 
group of isometries. 

\paragraph{}
The various paths to naturality make it a rich topic for learning about group theory. We hope this note
illustrates it. There are cute algebraic reasons for naturality like that a "simple group of cardinality not exceeding
the continuum is natural" for example. Cardinality at the moment still matters for some conditions, but it
might just be an artificial barrier. The removal of the cardinality constraint in the reflection group case shows this. 
We summarize some of the remaining obstacles in the last section. In the non-abelian case, a concrete example with unknown 
"naturality status" is the {\bf Heisenberg group} over a field $F$ of characteristic $3$, with 
$|F|>c$, where $c=2^{\aleph_0}$ is the {\bf cardinality of the continuum}. 
An other example is $L^{\infty}(X,SL(2,\mathbb{R}))$, if the probability space $X$ is non-standard. 
\footnote{In the case $L^{\infty}(X,{\rm PSL}(2,\mathbb{R}))$, the group is a reflection group which is
covered by \cite{Sutherland2026}. For standard probability spaces the structure theorem kicks in}.
In the abelian case, the group $C_3^{(k)}$ for cardinalities $k>c$ is not yet known to be natural.
We expect these cardinality restrictions to tumble, as encoding methods like in 
\cite{Babai1978} or transfinite induction as in \cite{Sutherland2026} should work.

\section{Groups generated by involutions}

\paragraph{}
An {\bf involution} or {\bf reflection} in a group $G$ is an element $g \neq 1$ 
such that $g^2=1$. We use the words involution and reflection as synonyms.
We call a group $(G,*)$ a {\bf reflection group} or {\bf involution group},
if it is a group with a subset $S$ of involutions 
such that every element $g \in G$ is a finite product of elements in $S$. 
{\bf Reflection group} is understood purely algebraically. No geometric realization like hyperplane 
reflections are assumed. The group $SO(3)$ for example is a reflection group, even so
the reflections that generate the group are line reflections. 

\begin{thm}[Reflection groups \cite{Sutherland2026} ]
Every reflection group is natural.
\end{thm}

\begin{proof}
The argument for $|G| \leq c$ appeared already in \cite{GraphsGroupsGeometry}. 
We do not go into the proof of \cite{Sutherland2026}. But it can be illustrative to sketch
the simpler case $|G| \leq c$ as it illustrates a central idea. 
Define the {\bf weighted Cayley graph} $\Gamma(G,S)$ equipped with 
geodesic metric generated by putting distinct weights $\lambda_s \in (1,2)$ on the edges
$(g,sg)$ (this uses the cardinality) and defining $d(g,sg)=\lambda_s$.
Right translations are now graph automorphisms as $d(ga,ha)=d(g,h)$ and
every $s \in S$ is metrically recognizable. 
If an isometry $T$ of $(\Gamma,d)$ fixes a vertex $g$, it must be the identity: 
the reason is that if $Tg=g$ then $T(s g)=s g$ and by induction on word length $T(g)=g$ 
for all $g \in G$. As right translation are always transitive, the action
is regular and ${\rm Iso}(G,d)=R(G)$. So much about the proof in the $|G| \leq c$ case. \\
A natural conjecture was then, and the naturality of $C_2^k$ for any cardinality supported 
that "every group $G$ generated by involutions is natural".
A few days after asking this on youtube and mentioning on our blog, the paper 
\cite{Sutherland2026} proved it with a quite astounding proof. 
\end{proof} 

\paragraph{}
An example of a finite reflection group is the symmetric group $S_n$. It is either generated by transpositions
or by its $(n-1)$ {\bf pancake flips}. Groups appearing as transposition games or
the square Rubik group. {\bf Weyl groups} are generated by finitely many reflections.
The statement $\langle S \rangle = G$ already means that
$G=\bigcup_n S^n$, meaning that every $g \in G$ is a product of 
finitely many $s_i \in S$. The set $|S|$ does not have to be finite. 
The case $|S|<\infty$ would mean {\bf finitely generated}. 

\paragraph{}
All real linear Lie groups generated by {\bf hyperplane reflections} 
are disconnected because the determinant of a hyperplane reflection is $-1$. 
This is not true if hyperplane reflections is replaced 
by reflections. An example is $SO(3)$ which is a connected reflection group
but where the generating set are line reflections.
The group $\mathbb{T}^1 \cong SO(2) \cong S^1$ is the simplest positive dimensional Lie group
that contains only one involution and so is not a reflection group. Also the abelian 
Lie group $\mathbb{T}^n$ is not a reflection group. 
The torus group is natural because connected Lie groups are natural in general.
Reflection groups can be connected in infinite dimensions. The unitary group 
$U(H)$ of an infinite dimensional Hilbert space illustrates this (the group is contractible). 
But $O(2) \times U(H)$ is an example of a reflection group in infinite dimensions that
is disconnected.

\paragraph{}
All orthogonal groups $O(n)$ and the pseudo-orthogonal groups $O(p,q)$ 
or {\bf complex orthogonal group} $O(n,\mathbb{C}) = \{ A \in M_n(\mathbb{C}),  A^T A=1 \}$
are reflection groups by the {\bf Cartan-Dieudonn\'e} theorem.
For the unitary group $U(n)$, the {\bf Cauchy-Binet determinant product formula} prevents the group 
to be a reflection group: involutions have determinant $1$ or $-1$ so that a unitary matrix 
with non-real determinant can not be the product of reflections. Cauchy-Binet implies that 
groups generated by reflections must have determinants $1$ or $-1$. But $U(n)$ is still
natural as it is a connected Lie group.

\paragraph{}
If $A$ is an arbitrary abelian group, then the {\bf generalized dihedral group} ${\rm Dih}(A)$ is always
a reflection group. Examples are $D_n={\rm Dih}(C_n)$ or $D_{\infty}={\rm Dih}(\mathbb{Z})$. 
But \cite{Sutherland2026} now makes sure that for any abelian group,
the generalized dihedral group $G={\rm Dih}(A) = A \rtimes_{-1} C_2$ is always natural.

\paragraph{}
For many manifolds $M$ like $S^n$ the groups ${\rm Diff}_0(S^n)$ or ${\rm Homeo}_0(S^n)$ 
are reflection groups. But the reflection group approach does not cover everything: 
${\rm Diff}_0(M)$ surfaces $M$ of genus larger than $2$ for example are not generated by reflections.
{\bf Mapping class group} ${\rm Diff}(M)/{\rm Diff}_0(M)$ issues already for $M=S^4$.
We will establish naturality of ${\rm Diff}(M),{\rm Homeo}(M)$ or ${\rm Sympl}( (M,\omega))$ 
by checking that it is not metabelian and that the cardinality bound works. The later
only requires the usual assumption, that the manifold $M$ is second countable.

\paragraph{}
The group $U(H)$ of unitary operators on an infinite dimensional Hilbert space 
$H$ is an infinite-dimensional {\bf Banach Lie group}.
By the theorem of Halmos-Kakutani \cite{HalmosKakutani}, every element in the unitary group 
$U(H)$ is a product of at most $4$ involutions, so that it is a reflection group. Interestingly,
their proof writes $U$ as a product of two shifts, then dihedrates each of the $\mathbb{Z}$ actions
to a $\mathbb{D}_{\infty}$ action. Also the group ${\rm GL}(H)$ of bounded invertible operators 
is a reflection group.

\paragraph{}
Let $(\Omega,\mathcal{A},P)$ be a standard probability space. 
The group of measure-preserving transformations ${\rm Aut}(\Omega,\mathcal{A},P)$ 
is a reflection group of {\bf reflections width} $\leq 3$ \cite{Fathi1978}. 
In the non-atomic case, $G$ is a non-abelian simple group as $G'=[G,G]$ is normal leads to $G'=G$
so that $G''=G' \neq 1$ and the group is perfect so that only one single involution needs 
to be present. For the automorphism group of a localizable measure algebra we still have
that $G=S \cup S^2 \cup S^3$ for some set of involutions but $|S|>c$ in general.

\paragraph{}
If $G,H$ are reflection groups then $G \times H$ is. More generally, for an
infinite family of groups we have $H=\oplus_{i \in I} C_2 \subset G=\prod_{i \in I} C_2$, 
with strict inclusion. Both are reflection groups. 
$H$ is generated by coordinate involutions, while $G$ is not necessarily generated by 
coordinate involutions. It is still generated by involutions like
$S=G \setminus \{1\}$ for example. Every element $g \in S$ is an involution. 
If there is a uniform bound on the {\bf reflection width} of reflection groups $G_i$ 
then $G=\prod_{i \in I} G_i$ is a reflection group too. 

\paragraph{}
Generalized dicyclic groups $G={\rm Dic}(A,y)$ that appear in the next section are never
reflection groups. The reason is that if $y=x^2$ and $x a x^{-1}=a^{-1}$ for all $a \in A$
prevents every element in $xA$ to be an involution so that involutions can only be present
in $A$ so that involutions only generate $A$. Dicyclic groups are the major enemy for
naturality. 

\paragraph{}
Assume $A$ is an abelian group. Define $A \to {\rm Dih}(A) = A \rtimes_{-1} C_2$ for
the {\bf generalized dihedral group}. This definition can not be extended in general to non-abelian groups
as $g \to g^{-1}$ is only an anti-automorphism then. For a non-abelian group $G$, the
semi-direct product $G \rtimes_{-1} C_2$ is not defined in general. Dihedration is also some
sort of {\bf Coxeterization} or ``make it Coxeter" construction. Rather than looking at the 
translation group $\mathbb{R}^n$ for example, take the group generated by point reflections.
Points are enhanced to become reflection operations. This is a notion which extends to 
{\bf Hadamard manifolds}, Riemannian manifolds where all sectional curvatures are non-negative.
In that case there is a nice reflection subgroup of ${\rm Diff}(M)$ generated by point reflections
$\{ s_p, p \in M\}$ enabled by the fact that between any two points, there is a unique geodesic. 

\begin{thm}
Every generalized dihedral group ${\rm Dih}(A)$ is natural.
\end{thm}

\begin{proof}
Any ${\rm Dih}(A)$ is generated by involutions and so a reflection group.
Every reflection group is natural. 
\end{proof}

\paragraph{}
We know already that the $p$-adic group of integers $\mathbb{Z}_p$ is natural if and only the prime $p$ is odd.
We can ``dihedrate'" it however: ${\rm Dih}(\mathbb{Z}_2)$ is natural.
Unlike $(\mathbb{Z}_2,+)$ that was not natural, the group $\mathbb{Z} \rtimes_{-1} C_2$ is natural.
The dihedration process works in general only for abelian groups $A$.
It is analog to the generalized dicyclic construction which also emerges
from an abelian group. Unlike for dihedration, the dicyclic enhancement process
depends on the choice of involution $y$.

\begin{center}
\begin{tabular}{|ccc|} \hline
Abelian  $A$        & Dihedrated natural group    &   Dicyclic, non-natural group  \\ \hline
$C_{2m}, m \geq 2$  & $D_{2m}$                  &   ${\rm Dic}_m$               \\
$C_4$               & $D_4$                     &   $Q_8={\rm Dic}_2$            \\
$C_4 \times C_2$    & $D_4 \times C_2$          &   $Q_8 \times C_2$ or $C_4 \rtimes_{-1} C_4$          \\
$R \times C_2$      & ${\rm Dih}(R\times C_2)$        &   ${\rm Dic}(\mathbb{R} \times C_2,(0,1) \cong \mathbb{R} \rtimes_{-1} C_4$. \\
$S^1$               & ${\rm Dih}(S^1,-1)={\rm Pin}^+(2)=O(2)$ & ${\rm Dic}(S^1,-1) \cong {\rm Pin}^-(2)$    \\
$C_{2^\infty}$      & ${\rm Dih}(C_{2^\infty})$       &   ${\rm Dic}(C_{2^\infty},y)$         \\ 
$\mathbb{Z}$        & $D_{\infty}$              &             - (not existent)                 \\
$\mathbb{Z}_2$      & $\mathbb{Z}_2 \rtimes_{-1} C_2$    &    - (not existent)                 \\ \hline
\end{tabular}
\end{center}
o
\paragraph{}
The last two examples in that table contain torsion free $A$, so that there is
no nonzero element with $2y=0$. The groups $Q_8 \times C_2$ and $C_4 \rtimes_{-1} C_4$ show
that unlike the dihedral construction which is canonical for $A$, 
the dicyclic construction can depend on the choice of involution $y$.
Any dicyclic construction can transition to a dihedral construction:
if $G={\rm Dic}(A,y)$ is non-natural, just replace the order-four element $x$ satisfying $x^2=y$
by an actual involution $z$. We get then ${\rm Dih}(A)=A \rtimes_{-1} C_2$ which is natural. 
Iconic examples in the table are ${\rm Pin}^-(2) = {\rm Dic}(S^1,-1)$ is non-natural, whereas
${\rm Pin}^+(2) = {\rm Dih}(S^1,-1) \cong O(2)$ is natural.
It is interesting that even so we currently do not know whether the
3-group $A=C_3^{k}$ with $k>c$ is natural or not. Still, we can say that
${\rm Dih}(A)$ is natural simply because it is a reflection group.

\section{Structure theorem for non-abelian groups}

\paragraph{}
Lets start with the definition of {\bf generalized dicyclic groups}.
Let $A$ be an arbitrary abelian group and let $y \neq 1$ be an involution in $A$.
Adjoin an element $x$ as a square root of $y$ so that $x^2=y, x a x^{-1} =a^{-1}$ for all $a \in A$. 
The group ${\rm Dic}(A,y) = \langle A,x, x^2=y, x a x^{-1} = a^{-1}, \forall a \in A \rangle$
is called a {\bf generalized dicyclic group}, if it is non-abelian. 
(The assumption of being non-abelian is standard but excludes $C_4$.)
Since $y \neq 1$ and $y^2=1$, the element $x$ has order $4$. 
The quaternion group ${\rm Dic}(C_4,-1)=Q_8$ is the smallest dicyclic group.
It uses $A=C_4 = \langle i \rangle = \{1,i,-1,-i\}$ and  $x=j$, a second quaternion
generator. Now $x^2=-1 \in A$ and $x^4=1$. 

\paragraph{}
Upgrading a theorem from \cite{GraphsGroupsGeometry}, we remove here the assumption that
the group $G$ is {\bf finitely generated}. We still need $|G| \leq c$, where $c=2^{\alpha_0}$ 
is the {\bf cardinality of the continuum}. We repeat the proof, as we see it as a
not so obvious improvement. It especially required to look at the literature more carefully.
The theorem is directly on the work of Leemann and de la Salle.
Their {\bf orientation rigidity theorem} (seen in the Appendix) is 
extracted from parts leading to Theorem~7 in \cite{LeemannDeLaSalle2021}. This theorem
did not need any cardinality assumptions (as they state explicitly). 
The cardinality bound $|G| \leq c$ is still 
needed still for the Cayley graph coloring and so for the construction of the
metric. But \cite{Sutherland2026} and earlier breaches of cardinality independent work
using clever encoding \cite{Babai1978}, suggest that the cardinality bound is 
not needed. This is open at the moment. We post this note now as the mathematical world 
looks already different next week. 

\begin{thm}[Structure theorem]
Let $G$ be a non-abelian group of cardinality $|G|\leq c$. \\
Then, $G$ is natural $\Leftrightarrow$ $G$ is not generalized dicyclic.
\end{thm}

\begin{proof}
$\Rightarrow$. We use contraposition and verify that if
$G$ is generalized dicyclic implies that $G$ is not natural. Indeed,
the group $G={\rm Dic}(A,y)$ admits an involution that fixes $A$ and flips 
$i(g)=g^{-1}$ for $g=xa \notin A$. This involution is an isometry.
The group $H = \{ R_x \circ i, R_a, a \in A \}$ produces a 
{\bf dihedral competitor} ${\rm Dih}(A)=A \rtimes_{-1} C_2$ 
not isomorphic to $G$. Since $|H|=2|A|=|G|$, it acts regularly on $G$. 
$H$ can not be isomorphic to $G={\rm Dic}(A,y)$ as $H={\rm Dih}(A)$ is 
a reflection group, while ${\rm Dic}(A)$ is not. Indeed, 
$I(H) = \{ g \in H, g^2=1\} = H$ but $I(G) = \{ g \in G, g^2=1 \} \neq G$.
The property $I(G)=G$ is a property for groups that survives isomorphisms.  \\

$\Leftarrow$ This part is harder. We show that if $G$ is not generalized dicyclic
then $G$ is natural. This requires us to construct a metric space which forces
the group structure. \\

(i) {\bf Definition of the metric space $(G,d)$.} This is where cardinality is
important still. The cardinality of all {\bf colors} $[g]=\{g,g^{-1}\}$
with $g \neq 1$ $\leq |G| \leq c$. It is possible therefore to
label each $\lambda_{[g]}$ with a distinct element in the open interval 
$(1,2) \subset \mathbb{R}$. Define the metric $d(x,x)=0$ and 
$d(x,y)=\lambda_{[y x^{-1}]} \in (1,2)$ for $x \neq y$. It is symmetric, 
because $[xy^{-1}]=[yx^{-1}]$. Triangle inequality follows because
for three distinct points $x,y,z$ we have $1 < d(x,z) < 2 < d(x,y)+d(y,z)<4$. 

(ii) {\bf Right translations are isometries.}
For $a \in G$, $(ya)(xa)^{-1}=yaa^{-1}x^{-1}=yx^{-1}$. This implies
that $d(xa,ya)=d(x,y)$. Hence, every right translation
$R_a:x \mapsto xa$ is an isometry and so belongs to ${\rm Iso}(G,d)$.

(iii) {\bf An isometry fixing $1$ preserves the colors $[g]$. } 
Consider the {\bf complete Cayley graph} on G, whose edge $\{x,y\}$
is given the color $[yx^{-1}]=\{yx^{-1},xy^{-1}\}$.
The metric d assigns to every such color its 
own distinct real number $\lambda([yx^{-1}])$. 
(Here is the bottle neck for the cardinality so far.) 
Consequently, every isometry of $(G,d)$ preserves each 
individual color. In particular, an isometry fixing the neutral 
element $1$ is a color-preserving automorphism of the complete 
Cayley graph fixing $1$.

(iv) {\bf Applying the Leemann–de la Salle's orientation-rigidity theorem.}
It states that a group is orientation-rigid 
precisely if it is neither generalized dicyclic 
nor abelian with an element of order greater than $2$.
As $G$ is non-abelian and not generalized dicyclic,
rigidity follows: the stabilizer of $1$ in the color-preserving automorphism 
group of the complete Cayley graph is trivial and ${\rm Iso}(G,d)_1=\{1\}$
follows. 

(v) {\bf The full isometry group agrees with right translations.}
Let $T \in {\rm Iso}(G,d)$, put $a=T(1)$.
Since $R_{a^{-1}}$ is an isometry, $R_{a^{-1}}\circ T$
fixes $1$. By (iv), it must be the identity and $T=R_a$. 
Therefore ${\rm Iso}(G,d)=\{R_a:a\in G\} \cong G$.
In particular, the isometry group acts regularly on $G$.

(vi) {\bf The metric determines the group structure.}
Suppose now that $(G,\circ)$ is any other group structure 
on the same underlying metric space $(G,d)$ and that 
every right translation $R_a^\circ: x \mapsto x\circ a$ 
is an isometry. The group $R^\circ(G)=\{R_a^\circ:a\in G\}$
is a subgroup of ${\rm Iso}(G,d)$, and 
it acts regularly, hence transitively, on G.

But ${\rm Iso}(G,d)=R(G)$ itself acts regularly. A transitive 
subgroup of a regular permutation group must be the whole 
regular group: indeed, the orbit of $1$ under a 
subgroup $H \leq R(G)$ has cardinality $|H|$, and 
it consists exactly of the elements represented 
by the corresponding translations. 
If the orbit is all of $G$, then $H=R(G)$
and $R^\circ(G)=R(G)$. Consequently, the group 
$(G,\circ)$ is isomorphic to the original group $G$. 
Hence $G$ is natural.
\end{proof}

\section{Applying the structure theorem}

\paragraph{}
To use the structure theorem, we looked for classes
of groups. We illustrate this with the classes of  
{\bf Non-Metabelian}, {\bf solvable}, {\bf simple}, 
{\bf Kazhdan}, {\bf ICC}, {\bf non-sofic} or 
{\bf Tarski Moster} groups. All of the following results
are just relying on definitions and are nice "soft" exercises
in group theory. We list them because we ourselves had 
fun with them.

\paragraph{}
The {\bf commutator group} $G'$ of $G$ is the group generated by 
all {\bf commutators} $[g,h] = g^{_1} h^{-1} g h$. Obviously
$G'=\{1\}$ is equivalent to the group being {\bf abelian}. 
If $G''=\{1\}$, the group $G$ is called {\bf metabelian}. One also uses the 
terminology that $G'$ is the {\bf derived group} and $G''$ is the 
{\bf second derived subgroup}. The quotient $G/G'$ is called 
the {\bf abelianization} of $G$. A group is called {\bf perfect} 
if $G'=G$. A group is called {\bf solvable} if some derived group $G^{(n)}$ 
is trivial. Obviously, a non-solvable group is non-abelian and so
non-metabelian. 

\begin{coro}
Every non-metabelian group of cardinality $\leq c$ is natural.
\end{coro}
\begin{proof}
Since generalized dicyclic groups are metabelian,
non-metabelian are not dicyclic. The structure
theorem applies.
\end{proof}
\paragraph{}
In other words, if a $G$ of not too large cardinality has 
an non-abelian commutator $G'$, then it is natural. We need only
to find four elements $a,b,c,d$ such that $[ [a,b],[c,d] ] \neq 1$. This handy test
is not necessary however as the example $G=S_3=D_3$ shows. 
The symmetric group $D_3$ is natural but $G''=1$ so that 
$[[a,b],[c,d]]=1$ for all $a,b,c,d$. The handy test
fails in many dihedral cases, like also for the infinite dihedral group
$G=D_{\infty}=C_2 * C_2$ which has 
$G'=\mathbb{Z}$ and $G''=\{1\}$. But $D_{\infty}$ is natural.

\paragraph{}
Here are some examples: if $G$ is not metabelian of cardinality smaller or equal than the
continuum, then $\oplus_{r \in \mathbb{R}} G$ is also not matabelian, and of 
cardinality not larger than the continuum. It is still non-natural.
The quaternion group $Q$ is non-abelian and metabelian. 
It is also not-natural. It is generalized dicyclic however.
The group $D_{\infty}$ is non-abelian and metabelian but natural.

\paragraph{}
Nice applications for the structure theorem 
are homeomorphism groups of manifolds, diffeomorphism groups of 
manifolds, syplectomorphism groups of symplectic manifolds 
or automorphism groups of Lebesgue probability spaces. In all cases, the 
structure theorem allows to bypass the sometimes subtle question, whether 
the group is a reflection group.

\paragraph{}
An other twist is to consider {\bf solvable groups}. Since generalized dicyclic groups 
$G={\rm Dic}(A,y)$ is metabelian, it is solvable. Reversing this gives

\begin{coro}
A non-abelian non-solvable group of cardinality $\leq c$ is natural. 
\end{coro}

\begin{proof}
A solvable group can not be generalized dicyclic.
\end{proof} 

\paragraph{}
Here is a cute remark which deals with one of the deepest mathematical
achievements of the last century, the {\bf classification of all finite simple groups}.

\begin{coro}
Every finite simple group $G$ is natural.
\end{coro}
\begin{proof}
If $G$ is abelian, then its order $|G|$ is prime and it is natural.
If $G$ is non-abelian, use that it is not solvable and so not-metabelian
and so not generalized dicyclic and by the structure theorem
not abelian.
\end{proof}

And similarly

\begin{coro}
Every simple group with $|G| \leq c$ is natural. 
\end{coro}
\begin{proof}
There are no infinite simple abelian groups. Simplicity 
implies non-metabelian and so not generalized dicyclic and by 
structure theorem, it must be natural.
\end{proof}

\paragraph{}
For a topological group $G$, {\bf Kazhdan’s property} means that every 
continuous unitary representation having almost invariant vectors must have 
a nonzero invariant vector. Given an abstract group, we say $G$ is 
{\bf discrete Kazhdan}, if it is Kazhdan for the discrete topology.
We gloss here over the details. The point is that Kazhdan implies non-solvable.

\begin{coro}
Infinite, discrete Kazhdan groups are natural.
\end{coro}
\begin{proof}
Infinite discrete Kazhdan groups are non-amenable and therefore non-solvable 
and so natural. An other thing to check is cardinarlity.
But a discrete group G with the Kazhdan property 
must be finitely generated.
If $Q$ a compact Kazhdan set in $G$ then the group $H=<Q>$ must be already $G$,
as otherwise $G$ would act on $l^2(G/H)$ by right translation.
\end{proof}

\paragraph{}
A group $G$ is called an infinite conjugate class group {\bf ICC} if every $g \in G, g \neq 1$ has 
an infinite conjugacy class $\{ hgh^{-1}, h\in G \}$. ICC groups must be non-abelian. 

\begin{coro}
Every ICC group $G$ of cardinality $\leq c$ is natural.
\end{coro}
\begin{proof}
We just need to know that ICC implies that it is not generalized dicyclic. 
Use contraposition and assume that $G$ is generalized dicyclic
$G={\rm Dic}(A,y)=\langle A,x | x^2=y, xax^{-1}=a^{-1} \forall a \in A \rangle$,
where $A$ is abelian and $y \in A$ has order $2$.
Take any $a \in A$. Conjugation by elements of $A$ does nothing to $a$, while conjugation by $x$ sends
$a \mapsto a^{-1}$ Hence its conjugacy class is contained in $\{a,a^{-1}\}$.
There are not infinitely many conjugacy classes. 
\end{proof}

\paragraph{}
Lets finally look at the property of {\bf sofic} which was in the news recently 
as there was a construction of a non-sofic group announced \cite{OpenAI2026TenAdvances}. 
There is now even a {\bf finitely generated non-sofic group}.
The pitch is that there is now a finitely presented group $G$ for which 
no sequence of finite labeled graphs can reproduce its Cayley graph 
locally on an asymptotically full proportion of vertices.

\paragraph{}
A group is called {\bf sofic} if it can be approximated well by finite symmetric groups $S_n$
in the following sense: for every finite $F \subset G$ and every $\epsilon>0$,
there is a map $\phi: F \to S_n$ such that $d(\phi(gh),\phi(g)\phi(h)) \leq \epsilon$,
whenever $g,h,gh$ are in $F$ and $d(\phi(g),1) \geq 1-\epsilon$. The metric $d$ on $S_n$
is the usual transformation metric on $S_n$: $d(S,T)=P[ \{ S(x) \neq T(x) \} ]$ with uniform
probability measure $P$ on $\{1, \dots, n\}$. 

\begin{coro}
Non-sofic groups are natural.
\end{coro}

\begin{proof}
Every metabelian group is solvable, every solvable group is amenable, and every
amenable group is sofic. So, If $G$ is non-sofic, then it is non-metabelian.
Since non-sofic groups are countable, hence natural.
\end{proof}
\paragraph{}
In the context of non-sofic groups, {\bf Tarski monsters} have
come up as candidates. These are infinite groups for which all subgroups have 
prime power order and if there exists a prime $p$ that every non-trivial proper subgroup of $T$
has order $p$. Despite that the subgroup lattice of a Tarski monster $G$ is very 
simple $1 \to C_p \to G$, they are still candidates for being non-sofic. But none has been found 
so far. 

\begin{coro}
Tarski monsters are natural.
\end{coro}
\begin{proof}
Tarski monsters are simple and so not-metabelian. 
They also can be generated by two elements: If $a,b$ are not in the same
subgroup of order $p$, then $H=\langle a,b \rangle$ can not be a proper 
nontrivial subgroup and so must be $G$. So the cardinality condition is satisfied.
\end{proof} 

\paragraph{}
The regular {\bf Kranz product} $G=A \wr B = A^{(B)} \rtimes B$ has become prominent
in the lamplighter group $C_2 \wr \mathbb{Z}$ or within the 
Rubik group which sits within $(C_3 \wr S_8) \times (C_2 \wr S_{12})$.

\begin{coro}
If $A,B \neq 1$, and $|A \wr B| <c$, then $A \wr B$ is natural.
\end{coro}
\begin{proof}
Check the structure theorem assumptions. The Kranz product is always 
nonabelian, the cardinality works. It can never be generalized dicyclic.
Case (i): If $B$ has an involution $b$, then
$(1,b) \in A \wr B$ is an involution, and it is non-central because $b$ moves coordinates in $A^{(B)}$. 
But every involution in a generalized dicyclic group is central. \\
Case (ii) If $B$ has no involutions but $A$ has an involution $a$, put a in a single coordinate 
of $A^{(B)}$. This gives an involution of $A\wr B$. It is again non-central because B moves that coordinate. \\
Case (iii) Finally, if neither $A$ nor $B$ has an involution, then $A \wr B$ has no involutions 
at all: an involution would project to an involution of B, so it must lie in the base, but the 
base has none. A generalized dicyclic group necessarily contains its distinguished central 
involution $y=x^2 \neq 1$ however. 
\end{proof}

\paragraph{}
An amusing example is $\mathbb{R} \wr \mathbb{R}=\oplus_{\mathbb{R}} \mathbb{R} \rtimes \mathbb{R}$ 
which has no involutions and which is metabelian. Its cardinality is $c$.
It still can not be generalized dicyclic because the later contain involutions.  

\paragraph{}
Let $A*B$ be the free product. It appears in examples like $F_2 = \mathbb{Z} * \mathbb{Z}$ or
$D_{\infty}= C_2 * C_2$ or $PSL(2,\mathbb{Z}) = C_2 * C_3$. 

\begin{coro}
The free product of two groups $1<|A|,|B| <c$ is always natural.
\end{coro}
\begin{proof}
A generalized dicyclic group has an abelian subgroup of index 2; in particular,
it is virtually abelian (contains an abelian subgroup of finite index). 
But a nontrivial free product $A*B$ is not virtually abelian, 
with one exceptional looking case: $C_2*C_2=D_\infty$ which we know already to be natural.
\end{proof}

\section{Abelian groups}

\paragraph{}
It is harder for an abelian group to be natural. We have seen
for example that the {\bf dyadic group of integers} $\mathbb{Z}_2$ is not natural 
while the {\bf $p$-adic groups of integers} $\mathbb{Z}_p$
are natural for odd primes $p$. We have also seen that any connected
Lie group is natural and so in particular also abelian Lie groups
like $\mathbb{T}^n$ or $\mathbb{R}^n$ even so they are not 
reflection groups. Lets start with restating the abelian 
structure theorem \cite{GraphsGroupsGeometry}:

\begin{thm}[Abelian structure theorem]
A finitely generated abelian group is natural $\Leftrightarrow$
$G$ is finite with $|G|$ odd or $G=C_2^r$.
\end{thm}

\paragraph{}
We do not repeat the proof as nothing has changed. There were two steps: \\
1) finite abelian is natural if and only if it $|G|$ is odd or $G=C_2^r$. \\
2) Infinite finitely generated abelian group $G$ are non-natural. \\

A nice example of a non-natural group is ${\rm Pin}^-(2) = \mathbb{T}^1 \times C_2
={\rm Pin}^-(2)={\rm Dic}(S^1,-1)$, The other case ${\rm Pin}^+(2)=O(2) = S^1 \rtimes C_2$
was generated by involutions and so was natural.

\paragraph{}
We have seen in \cite{GraphsGroupsGeometry} that the additive group of the 
$p$-adic field $\mathbb{Q}_p$ was natural: it was as an abstract group 
isomorphic to the group $\mathbb{R}$. A similar isomorphism argument can be shown for any 
group that is part of a vector space over $\mathbb{Q}$ for the which the 
cardinality is not larger than the continuum. 
For example, if $M$ is compact metric then $C(M,\mathbb{R}),+)$ is natural. As a 
torsion-free divisible abelian group, it is a $\mathbb{Q}$ vector space with 
${\rm dim}_{\mathbb{Q}} C(M,R) \geq {\rm dim}_{\mathbb{Q}}(R)=c$. Any nonzero or 
complex real vector space of cardinality $c$ is additively isomorphic to $\mathbb{R}$. 

\paragraph{}
We can use an {\bf affine rigidity} theorem of Jarosz \cite{Jarosz1988} and
combine it with Mazur-Ulam and $2$-divisibility to get:

\begin{thm}
Every real or complex vector space $(V,+)$ is natural.
\end{thm}

\begin{proof}
The {\bf Jarosz rigidity theorem} tells that 
every real Banach space admits an equivalent norm whose
surjective linear isometries are $\pm 1$ \cite{Jarosz1988}.
Put any norm on $V$, complete it to a Banach space $X$, 
apply Jarosz \cite{Jarosz1988}, to $X$ and restrict the resulting 
norm back to $V$. This produces a metric. 
By the {\bf Mazur–Ulam theorem}, every surjective isometry of $V$ is affine:
${\bf Iso}(V,d)=\{x \to a+x,x \to a-x: a \in V \}$.
Let $H$ is any regular subgroup of this isometry group. 
It cannot contain an isometry $x \to a-x$
because this has the fixed point $x=a/2$. This is $2$-divisibility.
The group $H$ consists entirely of translations. If it is transitive, 
it must contain every translation, so $H=R(V,+)$. 
Hence the additive group is natural.
\end{proof}

\paragraph{}
A consequence is a result on abelian Lie groups that does not require any cardinality assumption.

\begin{coro}
\label{Banachliegroup}
Every connected abelian Banach-Lie group is natural. 
\end{coro}

\begin{proof}
A connected abelian Banach–Lie group $G$ is the quotient
$X/\Gamma$ of a real Banach space $X$ by a discrete subgroup $\Gamma$.
The Jarosz-Mazur-Ulam construction in the proof of the previous theorem
gave the norm and so the metric on $X$. 
This defines a metric on $G$. Every isometry on $G$ 
lifts to an isometry on $X$, showing that any regular subgroup
of $G$ lifts to a regular isometry group of $G$ and so must be
the set of translations.
\end{proof} 

\paragraph{}
We have seen in \cite{GraphsGroupsGeometry},
that the additive group of the $p$-adic field $\mathbb{Q}_p$ was 
natural for all $p$. The reason was that it is group isomorphic to $\mathbb{R}$. 
A similar argument would show for example that if $M$ is compact metric space, 
then $C(M,\mathbb{R}),+)$ was natural. Actually, any 
finite-dimensional or countably dimensional $\mathbb{Q_p}$-vector space is natural.

\paragraph{}
For $K=\mathbb{R}$ or $\mathbb{Q}_p$, any same-cardinality extension 
immediately inherits naturality of its additive group. In particular,
$\mathbb{Q}_p(t), \mathbb{Q}_p(t_1,\ldots,t_n)$ are also additively isomorphic to $\mathbb{R}$, hence natural.
The interesting cases therefore begin when the base field is countable, such as
$\mathbb{Q}, \overline{\mathbb{Q}} i$ or $\mathbb{F}_p(t)$.
If $F$ is a field, we can look at $F(t)$, the {\bf field of rational functions}. 
In general, one can show that if $F$ is a field with $|F| \leq c$ and $(F,+)$ 
is natural, then $(F(t),+)$ is natural. Such considerations allow for 
a catchy conjecture: 

\begin{conj}
If $F$ is any field, then the additive group of $F$ is natural.
\end{conj}

\paragraph{}
It might be equivalent to the question whether $C_p^{k})$ is natural
for cardinalities larger than $c$ also. 
We know that the additive group of every division ring for characteristic 
$0$ and $2$ are natural. The characteristic $p$ case could
reduces to the previously unresolved question whether $C_p^{(k)}$ natural for 
every odd prime $p$ and every cardinal $k$.
In the case of finite fields $F_q$ of characteristic $p$, we again see
a rift between $p=2$ and odd primes $p$. 

\footnote{
The direct sum $\bigoplus_{i\in I} A_i = \{ (a_i)_{i \in I} \}$ has
$a_i=1$ for all but finitely many $i$, unlike 
the direct product $\prod_{i \in I} A_i = \{(a_i)_{i \in I} | a_i \in A_i\}$.
Cardinalities differ $|\bigoplus_{\mathbb{N}} C_2|=\aleph_0$, 
$|\prod_{\mathbb{N}} C_2 |=2^{\aleph_0}$. }

\begin{thm}
$V=C_p^{(k)}$ is natural for every prime $p$ and every $k \leq c$. 
\end{thm}

\begin{proof}
For $p=2$, the group $V$ is generated by reflections and part of the general 
result on reflections. For odd $p$ and finite $k$, it is part of the 
structure theorem of finite natural groups. For odd $p$ and $k \leq c$, we can
for every oriented pairs $\{v,-v\}$, $v\neq 0$ use distinct
$\lambda_{\{v,-v\}}\in(1,2)$ as $k \leq c$, and define
the metric $d(x,x)=0$ and $d(x,y) - \lambda_{y-x,x-y}$ if $x \neq y$. 
By the description of the color-preserving automorphisms of the complete 
Cayley graph, ${\rm Iso}(V,d) =R(V) \rtimes \langle i \rangle$
with $i(x)=-x$, every element outside $R(V)$ has the form
$x \to a-x$ with the unique fixed point $a/2$. Such an element therefore
cannot belong to a regular subgroup of ${\rm Iso}(V,d)$.
Consequently every regular subgroup is contained in $R(V)$, and
transitivity forces it to be $R(V)$. Also $V$ is natural.
\end{proof}

\paragraph{}
Leemann and de la Salle in \cite{LeemannDeLaSalle2021} cite the 
work of Babai \cite{Babai1978}. This work allows to remove the cardinality 
condition in the case $p=2$. The theorem of Babai \cite{Babai1978} tells
that every infinite group $G$ has a {\bf directed graph} whose full automorphism 
group is precisely the regular representation of $G$ showing so
that graphs of arbitrarily large cardinality can have 
extremely rigid automorphism groups. This allows to define a natural
metric on any elementary abelian 2 groups:

\begin{thm} 
The Boolean group $G=C_2^{(k)}$ is natural for every cardinal $k$.
\end{thm}

\begin{proof}
For finite cardinalities $k$, the situation is settled. In the infinite case,
Babai's theorem produces a directed Cayley graph $(G,E)=\Gamma(G,S)$ 
for some $S \subset G \setminus \{0\}$ for which ${\rm Aut}( (G,E) )=R(G)$.
Since in a Boolean group, every element is its own inverse, the 
Cayley graph is undirected and so a graphical regular representation. 
Now define from $(V,E)$ a metric taking only two nonzero distances:
Let $d(x,x)=0$, $d(x,y)=1$ if $(x,y) \in E$, and $2$ if $(x,y) \notin E$. 
This metric has exactly the original translations as isometries. 
As ${\rm Iso}(G,d) = R(G)$, the group $G$ is natural.
\end{proof}

\paragraph{}
For odd $p$, the preceding colored-metric argument stops at
$c$, there would be more Cayley colors to cover with. 

\begin{conj}
Also for odd $p$, $C_p^{(k)}$ is natural for $k>c$.
\end{conj}

\paragraph{}
For odd $p$ there are no nontrivial involutions. In this abelian case,
neither the reflection mechanism, nor Jarosz-Mazur-Ulam applies, nor Babai applies.
Leemann and de la Salle prove that every non-Boolean abelian group is
$G \rtimes \{1,i\}$ with $i(g)=g^{-1}$ with no cardinality restriction.

\paragraph{}
Adding a $C_2$-factor supplied the index-two co-set to construct a
competing affine regular action.

\begin{thm}
For nonzero cardinals $k,r$ and odd $p$, the group $G=C_2^r \oplus C_p^k$ is non-natural.
\end{thm}
\begin{proof}
Let $d$ be any metric for which all translations of $G$ are isometries. Since $G$ is abelian, the
inversion $\sigma(x)=-x$ is an isometry: translation invariance and symmetry give
$d(-x,-y)=d(0,x-y)=d(0,y-x)=d(x,y)$. Choose a nonzero homomorphism $\phi: C_2^k \to C_2$
and extend it trivially over the $C_p^k$-factor of $G$. Define $A={\rm ker}(\phi)$.
Then $[G:A]=2$ and $C_p^k \subset A$ and choose $b\in G \setminus A$.
The affine isometries $H= \{x \mapsto x+a, a \in A\} \cup \{x \mapsto b+a-x: a \in A\}$
form a subgroup of ${\rm Iso}(G,d)$. It is transitive because the orbit of $0$ is $A\cup(b+A)=G$.
It is also free. A translation by $a\neq 0$ has no fixed point. Proof: if $x\mapsto b+a-x$
had a fixed point, then $2x=b+a$ But $2G \subseteq C_p^k \subset A$
while $b+a\notin A$. Contradiction. We see that $H$ is a regular subgroup of the isometry group.
But $H \cong A \rtimes_{-1} C_2$ is non-abelian: choose $0 \neq u \in C_p^k \subset A$.
Since p is odd, $-u\neq u$, so conjugation by an element in the other co-set sends $u$ to $-u$. Hence
$H \not{\cong} G$, because G is abelian. Therefore every possible translation - invariant metric 
on $G$ admits a competing non-isomorphic regular group, proving non-naturality.
\end{proof} 

\paragraph{}
The same proof shows that $C_2^{(k)} \times A$ is non-natural, if $A$ is abelian but not 
elementary abelian of exponent $2$. Among abelian groups, the presence of a direct 
$C_2$-factor is a disaster for naturality: unless the entire group has exponent $2$, 
it forces non-naturality.
 Examples are $C_2^k \times \mathbb{Z}$,
$C_2^k \times \mathbb{Q}$,
$C_2^k \times \mathbb{R}$,
$C_2^k \times C_{p}^{\lambda}$
for odd $p$ and nonzero $\lambda$, regardless of cardinality. 
This nicely generalizes the earlier $\mathbb{R} \times C_2$ example.
The abelian assumption of $A$ was essential $C_2 \times S_3$ is natural (it is 
generated by involutions) while $C_2 \times Q_8$ is non-natural by the
finite classification theorem. 

\begin{thm}
For $|G| \leq c$ and abelian $A$, the group $G=A \rtimes_\phi C_2^k$ 
is natural as long as $\phi \neq 1$.
\end{thm}
\begin{proof}
$G$ is non-abelian but contains a non-central involution. 
Choose $s \in C_2^k$ acting non-trivially on $A$. Then
$s^2=1$ and $sas^{-1}=\phi_s(a) \neq a$ for some $a\in A$.
But in a generalized dicyclic group, every involution is central. 
Indeed, if $D={\rm Dic}(B,y)$, all elements outside the abelian 
index-two subgroup $B$ have square $y \neq 1$, 
so every involution lies in $B$, and an involution $u \in B$ satisfies $u^{-1}=u$,
hence the inversion action fixes it. It is so central.
Consequently $A \rtimes_\phi C_2^k$ with nontrivial action cannot be generalized dicyclic.
The classification theorem shows then the result.
\end{proof}

\section{More examples}

\paragraph{}
The reflection result can be applied to the group ${\rm Diff}(M)$ of some 
compact finite dimensional manifolds or to the identity component ${\rm Diff}(M_0)$. 
These groups have the cardinality of the continuum. Since they are simple groups,
naturality follows by the {\bf existence of a single involution}. 
The structure theorem bypasses the mapping class group difficulty:

\begin{coro}
For any manifold $M$, the group ${\rm {\rm Diff}}(M)$ is natural.
\end{coro}
\begin{proof}
We know that $|G| \leq c$ and that it is not metabelian. 
$G={\rm Diff}(M)$ is not metabelian because $G'={\rm Diff}_0(M)$, 
which is perfect so that $G''=G'$. 
\end{proof}

\paragraph{}
The situation for homeomorphism groups is similar. 

\begin{coro}
For any manifold $M$, the group ${\rm {\rm Homeo}}(M)$ is natural.
\end{coro}

For closed compact manifolds, the connected component of the 
group of homeomorphisms is perfect an can be written as 
products of local homeomorphisms. 
If $M$ is a compact topological manifold, then ${\rm Homeo}_0(M)$ is natural because
it is generated by involutions and because the group is simple. 
For ${\rm Homeo}(M)$ the structure of the {\bf mapping class group}in the smooth case
can differ from the topological case. 
While in dimensions $1,2,3$, the mapping class groups the groups are the same,
already in dimension 4, fancy Seiberg-Witten invariants matter.
In dimension 3, all finite groups can occur. 

\paragraph{}
Discrete subgroups of ${\rm PSL}(2, \mathbb{C})$ are called {\bf Fuchsian groups}. 
All non-natural Fuchsian groups are abelian like $C_2 \times C_3$. 
Some Fuchsian groups are finitely generated, other are not. This example was the reason 
to probe whether the finiteness assumption was really needed. 

\begin{coro}
Non-abelian Fuchsian groups are natural. 
\end{coro}

\begin{proof}
A Fuchsian group $G$ is a discrete subgroup of ${\rm PSL}(2,\mathbb{R})$. 
Since ${\rm PSL}(2,\mathbb{R})$ is second countable, every discrete 
subgroup is countable and $|G| \leq c$ applies.
No non-abelian Fuchsian group can be generalized dicyclic $D(A,y)$. 
As $G$ is a discrete subgroup of ${\rm PSL}(2,R)$, this is not possible. 
Case i) if A is infinite $Z$ this is not possible.
Case ii) If A is finite the G is finite and must be cyclic. 
All abelian ones must be subroups of $SO(2),R$ or $R^+$. 
An infinitely generated Fuchsian group is a free product of cyclic groups.
\end{proof}

\paragraph{}
There are 17 wall paper groups. There is only one abelian cases. It is 
$p_1$ is $\mathbb{Z}^2$ which we know not to be natural. An other example is 
$p3$, the orientation preserving Euclidean triangle group 
$D(3,3,3) = \langle x^3=y^3=z^3=xyz=1 \rangle \cong Z^2 \rtimes C_3$, 
where $C_3$ acts on the triangular lattice by rotation. Its commutator group
is $Z^2$ and there are no involutions.
Wall paper groups contain the translation lattice $Z^2$ and so are not Fuchsian. 
All 17 are finitely generated hence some of them are metabelian.
We can here look at the different paths to naturality:

\begin{coro}
All non-abelian wall paper groups are natural.
\end{coro}
\begin{proof}
It follows from the structure theorem as all 16 non-abelian wall paper group 
are not generalized dicyclic. 11 of them are non are metabelian and so natural.
Four of them are Euclidean reflection groups. 
\end{proof}

\paragraph{}
For example, $p3=D(3,3,3)$ is metabelian but contains no involutions at all.
The point group is the group obtained by factoring out translations. 

\begin{center}
\begin{tiny}
\begin{tabular}{|lccccc|} \hline
Group & Point group & Metabelian? & Reflection-generated? & Natural? & Reason \\ \hline
$p1$    & $1$   & yes & no  & no  & $\mathbb{Z}^2$ is abelian and non-natural \\

$p2$    & $C_2$ & yes & no  & yes & not generalized dicyclic \\
$pm$    & $C_2$ & yes & no  & yes & not generalized dicyclic \\
$pg$    & $C_2$ & yes & no  & yes & not generalized dicyclic \\
$cm$    & $C_2$ & yes & no  & yes & not generalized dicyclic \\
$pmm$   & $D_2$ & yes & yes & yes & reflections, not generalized dicyclic \\
$pmg$   & $D_2$ & yes & no  & yes & not generalized dicyclic \\
$pgg$   & $D_2$ & yes & no  & yes & not generalized dicyclic \\
$cmm$   & $D_2$ & yes & no  & yes & not generalized dicyclic \\
$p4$    & $C_4$ & yes & no  & yes & not generalized dicyclic \\
$p4m$   & $D_4$ & no  & yes & yes & reflections; non-metabelian \\
$p4g$   & $D_4$ & no  & no  & yes & non-metabelian \\
$p3$    & $C_3$ & yes & no  & yes & not generalized dicyclic \\
$p3m1$  & $D_3$ & no  & yes & yes & reflections; non-metabelian \\
$p31m$  & $D_3$ & no  & no  & yes & non-metabelian \\
$p6$    & $C_6$ & yes & no  & yes & not generalized dicyclic \\
$p6m$   & $D_6$ & no  & yes & yes & reflections, non-metabelian \\ \hline
\end{tabular}
\end{tiny}
\end{center}

\paragraph{} 
The group $pmm$ is the reflection group $D_\infty \times D_\infty$.
The finite weighted Cayley metric on $pmm$ does make it into a topological
group but left translation is not an isometry. 

\paragraph{}
As for space groups the situation is even easier. There is the abelian $p_1 \cong \mathbb{Z}^3$ but
all others are non-abelian and not generalized dicyclic. A nice example is 
the generalized dihedral ${\rm Dih}(\mathbb{Z}^3)=\mathbb{Z}^3 \rtimes_{-1} C_2$. 

\begin{coro}
All of the 229 non-abelian space groups are natural.
\end{coro}
\begin{proof} 
If $G$ is a space group and $T=\mathbb{Z}^3$ its translation subgroup.
If $G={\rm Dic}(A,y)$ then $A$ has index $2$ in $G$. But this would mean $T=A$ (*)
which contradicts that $y=x^2 \in A$ has order $2$ while $T$ is torsion free.
(*) Why is $T=A$? As part of the space group, every element $a \in A$ commutes with 
translations and so must be a translation so that $A=T$. 
\end{proof} 

\paragraph{}
And it is even stronger as no positive dimensional Euclidean crystallographic group
is dicyclic in any dimension $\geq 1$. 

\begin{coro}
No crystallographic group is dicyclic. 
So that all non-abelian crystallographic groups are natural.
\end{coro}

\paragraph{}
We always assume they a Lie group is second countable, 
so that $|G| \leq c$. We already knew that any connected Lie 
group was natural. We have also:

\begin{thm}
Every non-metabelian Lie group is natural.
\end{thm}

\begin{proof}
The cardinality condition $|G| \leq c$ holds. By the structure theorem, and 
because a non-abelian Lie group can not be generalized dicyclic, the naturality holds. 
\end{proof} 

\paragraph{}
Examples are the non-abelian Lie groups ${\rm Dic}(T^n,y)$, or 
${\rm Dic}(\mathbb{R}^k \times \mathbb{T}^n,y)$. They are 
dicyclic and insensitive to the above argument. 

\paragraph{}
Examples of abelian groups and generalized dicyclic groups: they are all non-natural
by the structure theorem. Generalized dicyclic groups are never reflection groups.
Here are some examples of generalized dicyclic groups:

\begin{center}
\begin{tabular}{|ccc|} \\ \hline
$A$         &  $y$                  & $G={\rm Dic}(A,y)$     \\  \hline
$S^1$       &  $-1$                 &   ${\rm Pin}^{-}(2)$     \\
$T^2$       &  $(-1,1)$             &   ${\rm Dic}(T^2,y)$     \\
$T^n$       &  $(-1,1,\ldots,1)$    &   ${\rm Dic}(T^n,y)$     \\
$\mathbb{R}^k \times T^n$  & $(0,-1,1,\ldots,1)$ & ${\rm Dic}(\mathbb{R}^k \times T^n,y)$ \\ \hline
\end{tabular}
\end{center}

\begin{coro}
Every non-natural Lie groups is disconnected and metabelian.
\end{coro}

\begin{proof}
(i) We now that connected Lie groups are natural. \\
    In the abelian case, this is even true for every connected abelian Banach Lie group.\\
(ii) Every non-metabelian Lie group is natural. \\
\end{proof} 

\section{Summary of open problems}

\paragraph{}
We had formulated a week ago the problem whether all reflection groups are natural. 
This is now settled \cite{Sutherland2026} with a surprising proof. We try here to summarize
a few open things. We use the notation $2A = \{ 2a | a \in A \}$ for Abelian groups, where
$2a=a+a$. We also use the notation $G^2 = \{ g^2 | g \in G \}$ for non-abelian groups,
where $g^2=g*g$.  Also use $\ker(2)=A[2]=\{a \in A | 2a=0 \}$ in the Abelian case.

\begin{itemize}
 \item Is $(F,+)$ natural for any field $F$? (It is true for fields $|F| \leq c$.)
 \item Is $C_3^{k}$ natural for every cardinality $k$? (A particular test case of the above).
 \item If $G$ is non-abelian, is:  $G$ natural equivalent to $G$ is not generalized dicyclic?
 \item Does for an abelian group $A$, the condition $A=2A$ imply naturality?
 \item Does for a non-abelian group $G$, the condition $G^2=G$ imply naturality?
\end{itemize}

\paragraph{}
The last two {\bf idempotence statements} were triggered by looking at already established
evidence so far:

\begin{itemize}
\item $\mathbb{R}$ is natural because $2 \mathbb{R} = \mathbb{R}$.
\item Odd finite Abelian groups are natural because $2 G = G$.
\item $C_2^k$ is natural for all cardinalities $k$. 
\item For abelian groups $A$ with $|A| \leq c$, $A$ is natural if and only if $2A=A$ or $A=C_2^k$.
\item For finite groups, we know that naturality is equivalent to $A=2A$ or $A=C_2^k$ (the reason 
is that $A=2A$ produces $C_{2m+1}$ groups which are natural and because of the Abelian
structure theorem. \\
\item For abelian $A$ with $2A \neq A$, there is an index $2$ subgroup $B$ which (if $A$ is not elementary abelian 2)
gives a dihedral competitor ${\rm Dih}(B)$ to $A$. 
\end{itemize}

\paragraph{}
Since all reflection groups are natural, the only real obstacle in the abelian case that remains is the
question {\bf "does $2A=A$  imply naturality?"} This would completely characterize all abelian groups
as it would settle the conjectured picture:

\begin{center}
Conjectured structure for Abelian groups: $A$ natural $\Leftrightarrow$ $A=2A$ or $A[2]=A$. 
\end{center}

\paragraph{}
As for non-abelian groups, the conjecture (following work of Leemann and de la Salle, especially 
their orientation rigidity theorem which equates generalized dicyclic with "orientation rigid") would
be:

\begin{center}
Conjectured structure for non-abelian groups: $A$ natural $\Leftrightarrow$ $A$ not generalized dicyclic.
\end{center}

\paragraph{}
The weaker conjecture for nonabelian groups that {\bf $G^2=G \Rightarrow$ natural} would support this general
conjecture. It would not settle it, as the class $G^2=G$ is smaller than the class of natural groups. 
(Example:the non-abelian free group $F_2 = \langle a,b \rangle$ is non-dicyclic - and so by the structure theorem natural -
but it is neither reflection generated, nor does it satisfy $F^2_2=F_2$). But it would be an important step towards
removing the cardinality restrictions. 

\paragraph{}
Like {\bf "reflection generated implies natural"}
(which is now freed from cardinality constraints), the {\bf idempotence conjecture} {\bf "$G^2=G$ implies natural"} would
be nice to have cardinality independent. In the abelian case, this is mirrored by the only remaining
obstacle {\bf $2A=A$ implies natural}.

\section*{Appendix: The orientation rigidity theorem of Leemann- de la Salle}

\paragraph{}
Let $G$ be an arbitrary group. With the set $S \subset G \setminus\{1\}$ generating $G$,
the Cayley graph $\Gamma(G,S)$ is called the {\bf complete Cayley graph}. Attaching
to each edge $(g,gs)$ the {\bf color} $[s]=\{s,s^{-1}\}$ produces the
{\bf complete colored Cayley graph}. A {\bf color-preserving automorphism} is a permutation 
$\phi:G \to G$ if $\phi(1)=1$ and $\phi(gs) \in \phi(g) \{s,s^{-1} \}$ for all $g \in G$ 
and $s \in S$. The group $G$ is called {\bf orientation-rigid} if $\phi={\rm Id}$ the 
only such permutation. The notion of orientation-rigid is defined for any abstract group.
\paragraph{}
For this appendix, GPT 5.6 sol helped us to read, analyze and extract the relevant part of 
their Theorem~7 in \cite{LeemannDeLaSalle2021} and to restrict the more general arguments 
to the complete Cayley graph case, we need here. 

\begin{thm}[Orientation rigidity theorem \cite{LeemannDeLaSalle2021}]
$G$ is orientation-rigid $\Leftrightarrow$ $G$ is neither generalized dicyclic nor 
abelian with an element of order greater than $2$.
\end{thm}

\paragraph{}
Central to the proof is the ability to detect the {\bf quaternion group} $Q_8$:

\begin{lemma}[ Quaternion-detection lemma of Leemann-DeLaSalle]
Let $g,h \in G$. If $gh=hg^{-1}$ and $hg=gh^{-1}$ then $\langle g,h \rangle$ is a quotient of $Q_8$. 
If $G$ is non-abelian, or contains an element of order greater than $2$, then $\langle g,h \rangle \cong Q_8$.
\end{lemma}
\begin{proof}
The first relation gives $hgh^{-1}=g^{-1}$. The second gives both $g=hgh$ and $g=h^{-1}gh^{-1}$. 
Combining them gives $g^2=h^2=h^{-2}$ so that $g^4=1$ and $g^2=h^2$. We see that $g,h$
satisfy the relations in the quaternion group $Q_8=\langle i,j | i^4=1, i^2=j^2, jij^{-1}=i^{-1} \rangle$
Because all proper quotients of $Q_8$ are elementary abelian $2$-groups.
\end{proof}

\begin{proof}
$\Rightarrow$ (\cite{LeemannDeLaSalle2021} credit \cite{Watkins1971,Watkins1976}).
If $G$ is abelian with an element of order greater than $2$, 
or if $G$ is generalized dicyclic, then the complete colored Cayley graph 
admits a nontrivial color-preserving automorphism fixing $1$. 
In fact the same automorphism preserves the inverse-pair colors for every 
symmetric Cayley graph of $G$.
Proof: If $G$ is abelian, inversion $i(g)=g^{-1}$ is a group automorphism.  
It is nontrivial precisely when $G$ has an element of order greater than 
$2$, and it preserves every inverse pair $\{s,s^{-1}\}$.
Now suppose $G={\rm Dic}(A,y)=A \sqcup xA$, where $A$ is abelian and $y \in A$ has order $2$
and $x^2=y, xax^{-1}=a^{-1}$ for $a \in A$. 
Define $\psi(a)=a, a\in A$ and $\psi(xa)=(xa)^{-1}, a \in A$.
Equivalently, since $x^{-1}=xy$ and $(xa)^{-1}=xay$, $\psi(g)=g$ if $g \in A$
$\psi(g)=gy$, $g \notin A$.
Because $y$ is central and the quotient $G/A \cong C_2$, this is a nontrivial 
group automorphism. For every $g\in G$ one has $\psi(g)\in\{g,g^{-1}\}$, 
so every inverse-closed set is invariant under $\psi$. 
Hence $\psi$ is color-preserving and fixes $1$ for every symmetric Cayley graph. \\

$\Leftarrow$ Every pair of distinct vertices is joined by an edge, colored by the inverse pair
$[yx^{-1}]=\{yx^{-1},xy^{-1}\}$. Suppose that $\phi$ is color-preserving and fixes $1$. 
Putting $g=1$ gives $\phi(g) \in \{g,g^{-1}\}$ for all $g \in G$. 
The problem is now algebraic. How can a permutation consistently choose between fixing and inversion,
element by element? Write $A=\{g\in G:\phi(g)=g \}$ and $B=\{g \in G:\phi(g)=g^{-1}\ne g\}$.
Then $G=A \sqcup B$ and every involution belongs to $A$, whereas every element of $B$ 
has order greater than $2$.  If $\phi \neq {\rm Id}$, then $B \neq \emptyset$. 
The converse first does two-generator calculations: \\

(A) {\bf Two fixed elements}: If $g,h \in A$ and $gh \in B$, then $\langle g,h \rangle \cong Q_8$ .
Proof: Since $B$ is closed under inversion, both $gh$ and $h^{-1}g^{-1}$ lie in $B$. 
Applying the color condition to the two factorizations give
$gh=hg^{-1}, hg=gh^{-1}$.  Moreover $gh\in B$ has order greater than $2$. 
The quaternion-detection lemma applies.

(B) {\bf Mixed pair}: If $g \in A$ and $h \in B$, then $hgh^{-1}=g^{-1}$.
If additionally $hg \in A$, then $\langle h,g \rangle \cong Q_8$.
Proof: If $hg \in A$, then $(hg)g^{-1}=h \in B$, so quaternion detection gives a copy of $Q_8$, 
and hence the second statement. If $hg\in B$, applying the color condition to 
$hg$ gives two possibilities. One directly yields the first statement, 
the other forces $g$ and $h$ to commute. The latter alternative, 
combined with the fixed/inverted condition again gives the first statement.

Since $G$ is non-abelian and $\phi \neq {\rm Id}$, we have $B \neq \emptyset$. Distinguish two cases:

{\bf (i): An abelian fixed part and a non-involution exists.} \\
Suppose the subgroup generated by the fixed elements is abelian and that there exists
$a_0 \in A, a_0^2 \ne 1$. Choose $x \in B$. By the mixed pair lemma,
$x a x^{-1} = a^{-1}, a \in A$.
If $x_1,x_2\in B$, then conjugation by $x_1 x_2$ fixes every element of 
$A$, because each $x_i$ acts there by inversion.  
In particular $x_1x_2$ centralizes $a_0$. 
It cannot lie in $B$, since the Mixed pair part above would then require it to invert $a_0$. 
Hence $x_1 x_2\in A$. Let $H$ be the subgroup generated by $A$ together with all products 
$x_1 x_2$ with $x_i \in B$. The same calculation as shows that $H$ is abelian, 
that $G=H\sqcup xH$, and that $xhx^{-1}=h^{-1}, h\in H)$. 
Also $x^2\in H$ and conjugation by $x$ fixes $x^2$, while inversion sends it to $x^{-2}$, so
$x^2=x^{-2}$. Since $x\in B$, $x^2 \ne 1$. Hence $x$ has order $4$, 
$y=x^2$ is a nontrivial involution of $H$, and $G={\rm Dic}(H,y)$.

{\bf (ii): The quaternion core} The remaining alternatives are
if $\phi$ agrees with inversion on the relevant products in the complete graph,
of if the subgroup generated by fixed elements is non-abelian.
In both instances, a non-commuting pair $i,j$ is produced for which $\langle i,j \rangle \cong Q_8$.
For instance, in the global-inversion case, color preservation gives for any $s,t\in G$ because
$st \in \{ ts,ts^{-1} \}$ and $ts \in\{ st,st^{-1} \}$.
The quaternion lemma assures that every pair either commutes or generates $Q_8$. 
Since $G$ is non-abelian, a quaternion pair exists.
Fix such a pair $i,j$. For each $s \in G$, multiply $s$ by one of $1,i,j,ij$
so as to cancel its conjugation action on $i$ and $j$. Set
$x_s=s$ if $[s,i]=[s,j]=1$, $x_s=is$  if $[s,i]=1, [s,j] \neq 1$, $x_s=js$ 
 if $[s,i] \ne 1,  [s,j]=1$ and $x_s=ijs$ if $[s,i] \neq 1, [s,j] \neq 1$.  \\
{\bf Normalization:} The element $x_s$ satisfies $x_s^2=1, x_s\in Z(G)$ for every $s\in G$. \\
\footnote{For more details, look at lemma $21$ in Leemann-de la Salle. Here is a sketch}
Proof: By construction $x_s$ commutes with $i$ and $j$. Apply the color condition 
to elements such as $j x_s$ and $ij x_s$. In the general setting of their paper, one must 
first verify that these elements belong to $S^{\leq 3}$, here that verification is automatic.
Comparing the two allowed color images forces $x_s=x_s^{-1}$, hence $x_s^2=1$. 
Applying the same two-generator analysis to $x_s$ and an arbitrary 
group element, shows that $x_s$ commutes with it. So, $x_s$ is a central involution. \\

Define $E= \langle x_s, s\in G \rangle$. The just established normalization shows that $E$ is a 
central elementary abelian $2$-group. Since every $s$ differs from $x_s$ by an element 
of $\langle i,j \rangle$, one has $G=\langle i,j,E \rangle$. Moreover $E \cap \langle i,j \rangle=\{1,i^2\}$
Splitting off this common central involution gives an elementary abelian $2$-group $H$ such that
$G\cong Q_8 \times H$. Such a group is generalized dicyclic: take
$A=\langle i \rangle \times H$, $y=i^2$ and $x=j$. 
Then $A$ is abelian of index $2$, and $j^2=i^2=y$, $jaj^{-1}=a^{-1}, a \in A$. 
We again end in a generalized dicyclic family.
\end{proof}

\bibliographystyle{plain}

\end{document}